\documentclass[12pt]{article}

\usepackage{ graphicx, amssymb, pdfpages, 
	amsmath,amsthm,color,mathtools,tikz}
\usetikzlibrary{arrows,calc,intersections}
\usepackage{hyperref,url}
\usepackage{tkz-euclide}

\newtheorem{lemma}{Lemma}[section]

\newtheorem{remark}[lemma]{Remark}

\newtheorem{example}[lemma]{Example}
\newtheorem{theorem}{Theorem}

\newtheorem{conjecture}{Conjecture}
\newtheorem{problem}[lemma]{Problem}

\begin{document}
\newcommand{\eps}{{\varepsilon}}
\newcommand{\proofend}{$\Box$\bigskip}
\newcommand{\C}{{\mathbb C}}
\newcommand{\Q}{{\mathbb Q}}
\newcommand{\R}{{\mathbb R}}
\newcommand{\Z}{{\mathbb Z}}
\newcommand{\RP}{{\mathbb {RP}}}
\newcommand{\CP}{{\mathbb {CP}}}
\newcommand{\Tr}{\rm Tr}
\newcommand{\g}{\gamma}
\newcommand{\G}{\Gamma}
\newcommand{\e}{\varepsilon}
\newcommand{\kk}{\kappa}
\newcommand{\pr}{\mathrm{pr}}

\title{On conformal points of area preserving maps and related topics}

\author{Peter Albers\footnote{
Mathematisches Institut,
Universit\"at Heidelberg,
69120 Heidelberg,
Germany;
peter.albers@uni-heidelberg.de}
 \and 
 Serge Tabachnikov\footnote{
Department of Mathematics,
Pennsylvania State University,
University Park, PA 16802,
USA;
tabachni@math.psu.edu}
} 

\date{\today}

\maketitle

\begin{abstract}
In this article we consider area preserving diffeomorphisms of planar domains, and we are interested in their conformal points, i.e., points at which the derivative is a similarity. We present some conditions that guarantee existence of conformal points for the infinitesimal problem of Hamiltonian vector fields as well as for what we call moderate symplectomorphisms of simply connected domains. We also link this problem to the Carath\'eodory and  Loewner conjectures.
\end{abstract}

\section{Intoduction} \label{sec:intro}

A {\it conformal point} of a diffeomorphism $F: \R^2\to\R^2$ is a point at which the derivative $dF$ is a similarity. In this article we consider area preserving diffeomorphisms of planar domains, and we are interested in their conformal points. Although such an area preserving diffeomorphism may be free of conformal points (see Example \ref{example_Panvov} below), we present some conditions on $F$ that guarantee their existence.

There are several motivations for this study. It was sparked by a question posed on MathOverflow \cite{MO1}, apparently motivated by the elasticity theory: {\it Is it true that an area preserving diffeomorphism of  the closed unit disc $D^2\subset \R^2$ possesses conformal points?}

We start with the infinitesimal version of the problem where a symplectomorphism is replaced by a Hamiltonian vector field. In terms of the Hamiltonian function $H(x,y)$, a point $(x,y)$  is conformal if two conditions on the second partial derivatives hold: 
\begin{equation} \label{eq:mainv}
H_{xx}(x,y)=H_{yy}(x,y),\ H_{xy}(x,y)=0.
\end{equation}
There are two ways to interpret  conditions (\ref{eq:mainv}). The first is to consider the Hessian
$$
\mathit{Hess}(H)=\begin{pmatrix}
H_{xx}&H_{yx}\\
H_{xy}&H_{yy}
\end{pmatrix}.
$$
This symmetric matrix has two orthogonal eigendirections, giving rise to two fields of directions, and a conformal point is a singular point of these fields,  occurring when the Hessian is a scalar matrix. 

This is similar to the fields of principal directions on a surface in $\R^3$. The singularities of theis field are the umbilic points of the surface, the points where the two principal curvatures are equal. This connection to classical differential geometry provides a  context and it  is one of our  motivations. 

The famous Carath\'eodory conjecture asserts that a smooth closed surface in $\R^3$ homeomorphic to the sphere has at least two {\it distinct} umbilic points. The sum of indices of the singular points of the field of principal direction equals 2, the Euler characteristic of $S^2$. Since the indices of singular points of line fields are half-integers, one concludes that the algebraic number of umbilic points equals 4. 
See Figure \ref{ellipsoid} for the umbilic points on a triaxial ellipsoid.

\begin{figure}[ht]
\centering
\includegraphics[width=.5\textwidth]{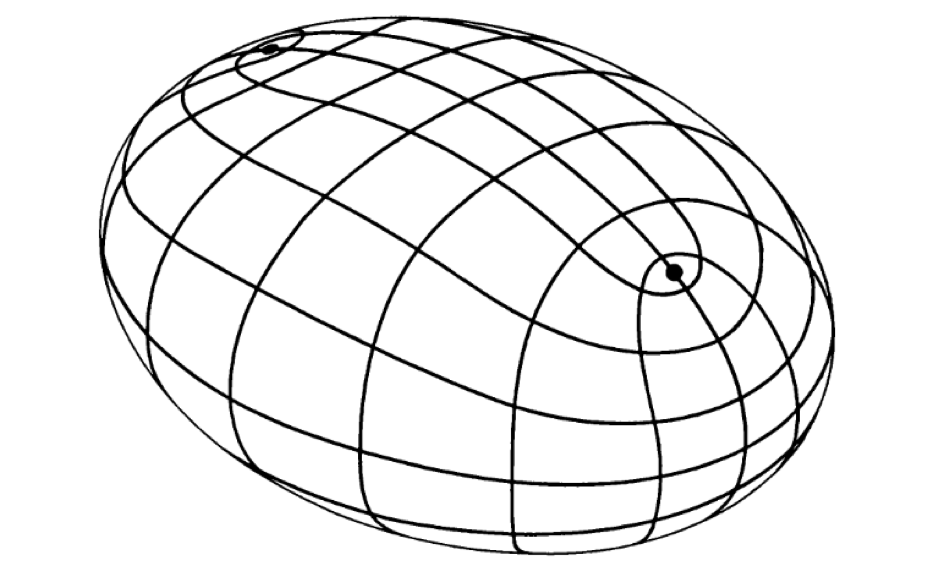}
\caption{The four umbilic points, two of which are visible, and the lines of curvature on an ellipsoid. This description is due to G. Monge, see \cite{SG} for a historical overview.
}
\label{ellipsoid}
\end{figure}

The  Carath\'eodory conjecture has a long and complicated history, see, e.g., \cite{GH,SG} and the references therein. The main approach to it  in the literature is local: conjecturally, the index of a singular point of the field of principal directions does not exceed 1; this conjecture implies the Carath\'eodory conjecture.
Examples of umbilics with every index not exceeding 1 are known.

Another way to interpret conditions (\ref{eq:mainv}) is to consider the vector field
\begin{equation} \label{eq:V}
V=(H_{yy}-H_{xx})\ \partial_x - 2H_{xy}\ \partial_y
\end{equation}
whose zeros are the conformal points. 
The components of the vector field  $V$ are the real and imaginary parts of $(\partial_x+i\partial_y)^2 H(x,y)$, where  $\partial_x+i\partial_y$ is the Cauchy-Riemann operator.

Loewner's conjecture concerns the indices of isolated zeros of planar vector fields $V_n$ whose two components are the real and imaginary parts of  $(\partial_x+i\partial_y)^n H(x,y)$, where $H(x,y)$ is a smooth function. We call these zeros {\it Loewner points}. Loewner's  conjecture states that the index of  an isolated zero of $V_n$ does not exceed $n$, see \cite{Tit}. 
For $n=1$, this is the well known statement that the index of an isolated zero of a gradient vector field is not greater than 1. 

In the case $n=2$, Loewner's conjecture is intimately related with the Carath\'eodory conjecture. In addition to differential geometry,  Loewner's conjecture is related with  hydrodynamics, see Section 3 of \cite{SX}. 
The connection of conformal points to the Loewner conjecture is another motivation for this study. 
\smallskip

Let us formulate our results. 

In Section \ref{sec:fields} we prove that if $H$ is a smooth function in a simply connected domain $D\subset \R^2$ such that $dH$ vanishes along the boundary and the system (\ref{eq:mainv}) has no solutions on $\partial D$, then there are two solutions inside $D$, counting with multiplicities. Under the same assumptions on $H$, we extend this result to conformal points of the Hamiltonian vector field of $H$ in a simply connected 2-dimensional domain with a general Riemannian metric. 

In Section \ref{subsec:L} we extend this result to Loewner points:
 if  a smooth function $H$ in a simply connected domain $D\subset \R^2$ is constant on the boundary, its first $n-1$ normal derivatives vanish, and the respective vector field $V_n$ has no zeros  on $\partial D$, then there are $n$ Loevner points inside $D$, counting with multiplicities. 

In Section \ref{sec:moderate} we consider conformal points of symplectomorphisms of simply connected domains $D\subset \R^2$. We prove that if a symplectomorphism preserves $\partial D$ pointwise and is {\it moderate} then it  possess two conformal points, counted with multiplicity. The  definition of a moderate symplectomorphism is given in Section \ref{sec:moderate}; informally speaking, a moderate symplectomorphism is sufficiently close, with its first derivatives, to the identity map. 

In this sense, this result is akin to the first results on the existence of symplectic fixed points, before the advent of the pseudo-holomorphic curves and the Floer homology. We wonder whether the existence of conformal points can be tackled using the techniques of the modern symplectic topology.  

Section \ref{sec:compl} concerns various interpretations of the Carath\'eodory conjecture and its  relation to conformal points. 
We show that, in a proper sense, umbilic points are conformal points of a Hamiltonian vector field on $S^2$, and we conjecture that an area preserving diffeomorphism of $S^2$ possesses at least two distinct conformal points.

\section{Conformal points of Hamiltonian vector fields} \label{sec:fields} 

Let $D\subset \R^2$ be a connected and simply connected domain $D$ with non-empty interior, bounded by a smooth closed curve $\gamma$. Let $H: D \to \R$ be a smooth function defined in a neighborhood of $D$ and which is constant on $\gamma$. Thus, its Hamiltonian vector field $X_H=H_y \partial_x-H_x \partial_y$ is tangent to the boundary $\gamma$. 
We  ask if the time-$\e$-flow $F_\e:D\to D$ of $X_H$ has conformal points. That this is not true without further assumptions is illustrated by the following example suggested to us by D. Panov.

\begin{example}\label{example_Panvov} 
{\rm Let $\g$ be the ellipse $x^2+cy^2=1$ and $H(x,y)=x^2+cy^2$. If $c\ne 1$, the Hamiltonian vector field $H_y \partial_x-H_x \partial_y$ has no conformal points, i.e.,~the 
infinitesimal symplectomorphism given by
$$
F: (x,y) \mapsto (x+\e H_y, y-\e H_x)=(x+2\e cy, y-2\e x).
$$
has no conformal points. By continuity, it follows that, for sufficiently small $\e\neq0$, the  time-$\e$-flow of this Hamiltonian vector field is free from conformal points as well.
}
\end{example}

This example points towards a difference between this question and say a fixed-point problem. Indeed, the time-$\e$-flow $F_\e$ of an autonomous vector field $X$ is fixed-point free for all sufficiently small $\e>0$ if and only if $X$ has no zeros. However, it is unclear to us if the existence of a conformal point for the Hamiltonian vector field forces the existence of conformal point for the time-$\e$-flow. Nevertheless, the infinitesimal question is interesting on its own right and links to the Carath\'eodory and Loewner conjectures as we mentioned in the introduction. We return to the question for time-$\e$-flows in Section \ref{sec:moderate}.

\begin{problem}
Is there an example of a Hamilton function $H:D\to\R$ such that the Hamiltonian vector field has a conformal point but for all sufficiently small $\e\neq 0$ the time-$\e$-flow of $H$ has no conformal point?
\end{problem}

Let us now consider a general infinitesimal symplectomorphism
$$
F: (x,y) \mapsto (x+\e H_y, y-\e H_x)
$$
generated by $H:D\to\R$ with $H$ being constant on $\g$. The Jacobian of $F$ is
$$
\begin{pmatrix}
	1+\e H_{xy}&\e H_{yy}\\
	-\e H_{xx}&1-\e H_{xy}
\end{pmatrix}.
$$
$F$ has conformal points, by definition, if this Jacobian matrix is an infinitesimal similarity, that is, if and only if 
\begin{equation*} \label{eq:maineq}
H_{xx}=H_{yy},\ H_{xy}=0,
\end{equation*}
(equation (\ref{eq:mainv}), i.e., if the Hessian $\mathit{Hess}(H)$ of $H$ is a scalar matrix. 

Let $\g(t)$ be an arc length parameterization of the boundary of $D$, and let $k(t)$ be the curvature of $\g(t)$. Assume that $\g$ is oriented as the boundary of $D$ and 
let $n(t)$ be the inward unit normal vector field along $\g$. A sufficiently small neighborhood of $\g$ is parameterized as $\g(t)+sn(t), s\in (-\eps,\eps)$. Near $\g$ we write  the Hamiltonian function as $H(t,s)$.

\begin{theorem} \label{thm:two}
Let the function $H:D\to\R$ be constant on the boundary $\g$ of $D$. We assume that the system of equations \eqref{eq:mainv} has isolated solutions in $D$, none of which lie on the boundary of $D$. If $H$ satisfies one of the following  conditions
\begin{enumerate}  \renewcommand{\theenumi}{\roman{enumi}}\renewcommand{\labelenumi}{\theenumi)}  \itemsep=0ex
\item $H_{ss}(t,0)+k(t) H_s(t,0)$ does not vanish on $\g$,
\item $H$ has zero normal derivative along $\g$,
\end{enumerate}
then \eqref{eq:mainv} has  two solutions in the interior of $D$, counting with multiplicities.
\end{theorem}

Before giving a proof, we make two comments. First, since $H$ is constant on $\g$, one can write $H(t,s)=c+sh(t,s)$, where $h$ is a smooth function. Then along the boundary $\g$ of $D$
$$
H_{ss}(t,0)+k(t) H_s(t,0) = 2h_s(t,0)+k(t)h(t,0).
$$
If $\g$ has positive curvature, that is, $k>0$, one can always add a constant to $h$ to ensure that $H_{ss}(t,0)+k(t) H_s(t,0)>0$ for all $t$. 


Second, since $H$ is constant on $\g$, having zero normal derivative is equivalent to $dH=0$ on $\g$, that is, to the Hamiltonian vector field $H_y \partial_x-H_x \partial_y$ being identically zero on the boundary. 

\begin{proof}
We consider the vector field 
\begin{equation*} \label{eq:V}
V=(H_{yy}-H_{xx})\ \partial_x - 2H_{xy}\ \partial_y
\end{equation*}
(equation (\ref{eq:V}). By assumption, $V$ has finitely many zeros in $D$, none of which lie on the boundary. We want to show that $V$ has at least two zeros inside $D$, counted with multiplicities. For that purpose we compute the vector field $V$ along $\g$ in $(t,s)$-coordinates.

Let $\alpha(t)$ be the direction of $\g(t)$, that is, $\g_t=(\cos\alpha,\sin\alpha)$. Then $k=\alpha_t$ and  $n=(-\sin\alpha,\cos\alpha)$. Combining the assumption that $H$ is constant on $\g$, i.e., $H_t(t,0)=H_{tt}(t,0)=0$, with a calculation using the chain rule yields
\begin{equation*}
\begin{split}
V(t,0) = \;\;\;\;&[(H_{ss}(t,0)+k(t) H_s(t,0))\cos 2\alpha + 2 H_{st}(t,0) \sin 2\alpha]\ \partial_x \\[.5ex]
+&[(H_{ss}(t,0)+k(t) H_s(t,0))\sin 2\alpha - 2 H_{st}(t,0) \cos 2\alpha]\ \partial_y,
\end{split}
\end{equation*} 
or, in complex notation,
$$
V=(H_{ss}+k H_s - 2i H_{st})  e^{2i\alpha}.
$$
Under the first assumption, i.e.,~if $H_{ss}(t,0)+k(t) H_s(t,0)\neq 0$  for all $t$, one has that  $V$ has non-zero inner product with the vector field $\cos(2\alpha)\partial_x + \sin(2\alpha)\partial_y$. The latter has rotation number 2 along $\g$, hence the rotation number of $V$ also equals 2. The Poincar\'e-Hopf theorem implies that $V$ has two zeros inside $D$, counting multiplicities.

Under the second assumption, i.e.,~if $H$ has zero normal derivative along $\g$, then $H_s(t,0)=H_{st}(t,0)=0$, and $V=H_{ss} e^{2i\alpha}$. By assumption, $V$ has no zeros on the boundary, i.e., in this case, $H_{ss}$ does not vanish on the boundary. In particular, $V$ has again rotation number 2 along $\g$ and we conclude as above finishing the proof of Theorem \ref{thm:two}.

Let us present an alternative argument in the second case. Instead of $V$ we consider the Hessian matrix $\mathit{Hess}(H)$. At a point where this matrix is not scalar, i.e.,~at a non-conformal point, it has two distinct eigenvalues with orthogonal eigendirections. One may distinguish between the two eigendirections by the value of the eigenvalues. Thus one obtains two orthogonal line fields whose common singularities are the conformal points. 

The rotation number of a line field along a closed curve is its number of the complete turns divided by 2. In particular, the indices of singular points of line fields are half-integers. This convention is consistent with the rotation number for vector fields: when a non-oriented line makes a full turn, the respective oriented line makes only half a turn. See Figure \ref{ind} for a singularity of a line field with index $\tfrac12$.

\begin{figure}[ht]
\centering
\includegraphics[width=.25\textwidth]{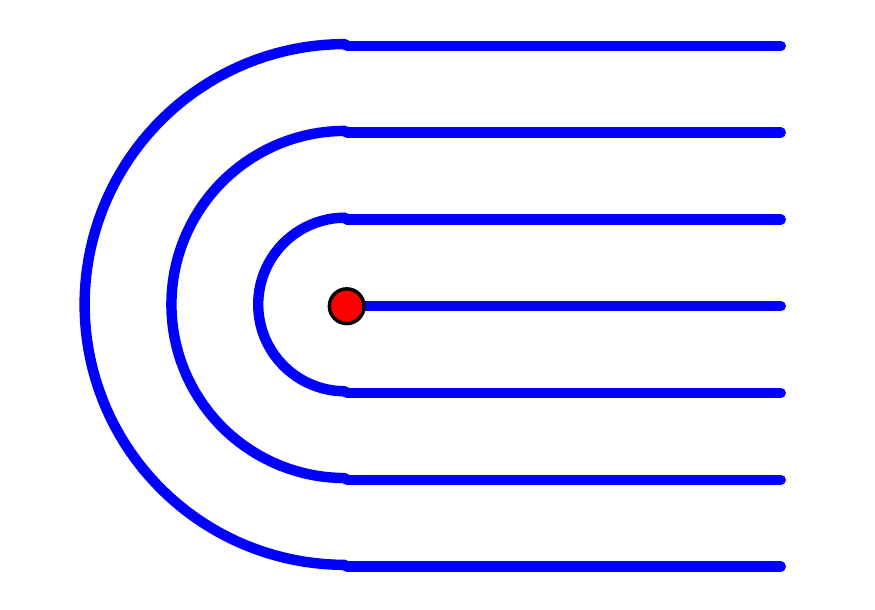}
\caption{A singularity of a line field with index $\tfrac12$.}
\label{ind}
\end{figure}

The Poincar\'e-Hopf theorem holds also for line fields: the rotation number on the boundary of $D$ equals the sum of indices of its singularities inside $D$. 

As explained above, the second assumption is equivalent to $dH=0$ on the boundary curve $\g$, i.e.,~$\g$ is a critical curve of the function $H$. Therefore the tangent line to $\g$ is in the kernel of $\mathit{Hess}(H)$, hence it is an eigendirection. If there are no conformal points on the boundary, then $\g$ is an integral curve of one of the two fields of eigendirections of $\mathit{Hess}(H)$. It follows that the rotation number of this field around $\g$ is 1, and by the Poincar\'e-Hopf theorem it has to have singular points  inside $D$. Since their indices are half-integers, one has at least two such points, counted with multiplicities.
\end{proof}

\begin{remark}
{\rm We saw that the vector field $V$ rotates twice as fast as the  field of eigendirections of $\mathit{Hess}(H)$. This fact is known in the literature on the Carath\'eodory conjecture, see, e.g., \cite{SX}. 
}
\end{remark}

The next example shows that, under the assumptions of Theorem \ref{thm:two}, one may have exactly one conformal point (necessarily of multiplicity 2). We also discuss a limitation of the above argument if we allow for conformal points on the boundary of $D$.

\begin{example}
{\rm Let $D$ be the closed unit disc, and let $H(\alpha,r)=r^a-\frac{a}{2} r^2$, where $(\alpha,r)$ are polar coordinates on $D$ and $a>2$. This function is at least of class $C^2$ (and if $a$ is an even integer, it is a polynomial), it is constant and has zero normal derivative on the boundary $r=1$.

A calculation similar to the one in the proof of Theorem \ref{thm:two} yields
$$
V = \left(H_{rr}-\frac{1}{r}H_r \right) (\cos 2\alpha,\sin 2\alpha) = a(a-2)r^{a-2} (\cos 2\alpha,\sin 2\alpha).
$$
This vector field $V$ is non-singular in the punctured disc, and its only zero is the origin. 

We now consider $H$ as a function on  $D':=\{z\in \C\mid |z-1|\leq 1\}$, the unit disk around $1\in\C$. Then $H$ has no conformal points in the interior but one of multiplicity 2 on the boundary. Moreover, the vector field $V$ has precisely one zero which lies on the boundary. 

As in the alternative argument above, one is tempted to consider the singular line field generated by the vector field $V$ (with one singularity in $0\in \partial D'$) since it has a continuous extension to a non-singular line field along the boundary, namely the line field generated by the vector field $(\cos 2\alpha,\sin 2\alpha)$. Now the conclusion from the Poincar\'e-Hopf theorem seems to be that our example necessarily has two conformal points in the interior of $D'$ which is wrong as explained above. The error lies in the fact that, while the line field indeed has a non-singular continuous extension to the boundary given by $(\cos 2\alpha,\sin 2\alpha)$, there is no continuous extension to all of $D'$. 

It is unclear to us what the correct statement of Theorem \ref{thm:two} is if we allow for zeros of $V$, i.e.,~conformal points, on the boundary.
}
\end{example}

\begin{problem}
Is there a Hamiltonian function $H:D\to\R$ such that the Hamiltonian vector field has only one conformal point with multiplicity 1 on the boundary of $D$?
\end{problem}

The circle of questions makes sense on other surfaces. Since there are no topological obstructions on the torus coming from the Poincar\'e-Hopf theorem we pose the following problem.

\begin{problem}
Let $H(x,y)$ be a function on the torus $\R^2/L$ where $L$ is a lattice. Does the respective Hamiltonian vector field necessarily have conformal points? What about Hamiltonian symplectomorphisms? 
\end{problem}

We point out that elements from $SL(2,\Z)$ acts on $\R^2/\Z^2$ by area-preserving maps and, except for the $\pm$identity, none of them have conformal points. However, again except for the identity, none of them are Hamiltonian diffeomorphisms, i.e., generated by a Hamiltonian function on $\R^2/\Z^2$.

Finally, we extend case ii) of Theorem \ref{thm:two} to domains $D\subset \R^2$ as above but equipped with a general Riemannian metric. This metric induces a symplectic form and a conformal structure, and we may again consider conformal points of Hamiltonian vector fields. 

\begin{theorem} \label{thm:twoR}
Let $dH=0$ on the boundary $\g$ so that the Hamiltonian vector field $X_H$ of $H$ has no conformal points on $\g$. Then $X_H$ has two conformal points in the interior of the domain, counting with multiplicities.
\end{theorem}

\begin{proof}
Let the metric be given in the conformal form $\tfrac{1}{g(x,y)} (dx^2+dy^2)$, where $g$ is a positive function. Then the  symplectic form is $\tfrac{1}{g(x,y)}dx \wedge dy$, and therefore the Hamiltonian vector field of a function $H(x,y)$ is $g(x,y) (H_y \partial_x - H_x \partial_y)$. A point $(x,y)$ is conformal for the Hamiltonian vector field if and only if
\begin{equation*} \label{eq:Riem}
(gH_x)_x=(gH_y)_y,\ (gH_y)_x=-(gH_x)_y.
\end{equation*}
As before, consider the vector field
$$
V=[(gH_y)_y-(gH_x)_x]\ \partial_x -[(gH_x)_y+(gH_y)_x]\ \partial_y.
$$
If $dH=0$ on the boundary then, on $\g$, one has 
$$
V=g [(H_{yy}-H_{xx})\ \partial_x - 2H_{xy}\ \partial_y].
$$
This field still has rotation number 2 along $\g$, and the rest of the proof is the same as that of Theorem \ref{thm:two}.
 \end{proof}

\section{Loewner points} \label{subsec:L}

Consider the vector fields $V_n$ whose components are the real and imaginary parts of  the function 
$(\partial_x+i\partial_y)^n H(x,y)$.  One has $V_1= \nabla H$, and $-V_2$ is the vector field $V$ that we considered above.
Recall that  zeros of $V_n$ are termed Loewner points. 

We extend case (ii) of Theorem \ref{thm:two} to Loewner points. Let $D$ be as in Theorem \ref{thm:two}.

\begin{theorem} \label{thm:n} 
Assume that the function $H$ is constant on the boundary $\g$ and  has 
vanishing $j$-th normal derivatives along $\g$  for  $j=1,\ldots,n-1$. Moreover, assume that  
the vector field $V_n$ has only isolated zeros in $D$ and no zeros on the boundary.
Then $V_n$ has $n$ zeros in the interior of $D$,  counting with multiplicities.
\end{theorem}

\begin{proof}
We use the same coordinates $(t,s)$ as before, and we again calculate the vector field $V_n$ along $\g$. The chain rule implies that
$$
\partial_x+i\partial_y = e^{i\alpha(t)} \left(\frac{1}{1-sk(t)}\ \partial_t + i\partial_s\right)=:e^{i\alpha(t)} {\mathcal{D}}
$$
where we recall that $\alpha_t=k$. Moreover, we observe
\begin{equation}\nonumber
\begin{aligned}
[\mathcal{D},e^{in\alpha(t)}]f&=\mathcal{D}(e^{in\alpha(t)}f)-e^{in\alpha(t)} \mathcal{D}f\\
&=(\mathcal{D}e^{in\alpha(t)})f+e^{in\alpha(t)}\mathcal{D} f-e^{in\alpha(t)} \mathcal{D}f\\
&=ne^{in\alpha(t)}\frac{i\alpha_t(t)}{1-sk(t)}f,
\end{aligned}
\end{equation}
that is, the first-order differential operator $\mathcal{D}$ satisfies the commutation relation
$$
[{\mathcal{D}},e^{in\alpha(t)}]=ne^{in\alpha(t)} \frac{ik(t)}{1-sk(t)} =:ne^{in\alpha(t)} F(t,s),
$$
where the right hand side is the zero-order operator of multiplying by a function. Lemma \ref{lem:commutation}  below then implies that
$$
(\partial_x+i\partial_y)^n = e^{ni\alpha}({\mathcal{D}}+(n-1)F)({\mathcal{D}}+(n-2)F)\ldots ({\mathcal{D}}+F){\mathcal{D}}.
$$
Using the specific form of $\mathcal{D}$ one has
$$
({\mathcal{D}}+(n-1)F)({\mathcal{D}}+(n-2)F)\ldots ({\mathcal{D}}+F){\mathcal{D}} H(t,s)=\sum_{1\le j+k\le n} c_{jk}(t,s) \partial_t^j \partial_s^k H(t,s),
$$
where $c_{jk}(t,s)$ are some functions. Evaluating the right hand side on the boundary $\g$ and using the assumption that the first $n-1$ normal derivatives of $H$  vanish on $\g$, i.e.,~$\partial_s^k H(t,0)=0$ for $j=1,\ldots,n-1$, the only term that survives is $i^n (\partial_s^n H)(t,0)$. It follows that on the boundary one has
$$
V_n = i^n (\partial_s^n H)(t,0)\ e^{in\alpha(t)},
$$
and the proof concludes the same way as that of  Theorem \ref{thm:two}. 
\end{proof}

\begin{lemma}\label{lem:commutation}
Let $\mathcal{D}$ be a first order differential operator on real-valued functions in several variables. We assume that $\mathcal{D}$ satisfies the commutation relation $[\mathcal{D},a^n]=na^nb$, $n\in\mathbb{N}$, for some functions $a$ and $b$, where $a^n$ resp.~$na^nb$ are understood as multiplication operators. Then 
\begin{equation}\nonumber
(a\mathcal{D})^n=a^n\prod_{k=1}^n(\mathcal{D}+(n-k)b)
\end{equation}
holds for all $n\in\mathbb{N}$.
\end{lemma}

\begin{proof}
We proceed by induction. The claim is clearly true for $n=1$. For the induction step let $f$ be some function. To simplify the following computation we set $g:=\prod_{k=1}^n(\mathcal{D}+(n-k)b)f$. Then
\begin{equation}\nonumber
\begin{aligned}
(a\mathcal{D})^{n+1}f&=(a\mathcal{D})(a\mathcal{D})^nf\\
&=a\mathcal{D}\left(a^n\prod_{k=1}^n(\mathcal{D}+(n-k)b)f\right)\\
&=a\mathcal{D}(a^ng)\\
&=a\big(a^n\mathcal{D}g+na^nbg\big)\\
&=a^{n+1}(\mathcal{D}+nb)g\\
&=a^{n+1}(\mathcal{D}+nb)\prod_{k=1}^n(\mathcal{D}+(n-k)b)f\\
&=a^{n+1}\prod_{k=0}^n(\mathcal{D}+(n-k)b)f\\
&=a^{n+1}\prod_{k=1}^{n+1}(\mathcal{D}+(n+1-k)b)f\\
\end{aligned}
\end{equation}
completes the proof by induction. We used the induction assumption in the second equation and the commutation relation in the fourth equation.
\end{proof}

\section{Conformal points of moderate symplectomorphisms} \label{sec:moderate}

We now consider an extension of the second case of Theorem \ref{thm:two} to symplectomorphisms of a domain in $\R^2$ satisfying a certain condition  to be defined below.

We continue to assume that $D$ is as above, that is, a connected and simply connected domain $D$, 
bounded by a smooth closed curve $\gamma$. Consider a symplectomorphism $F:D\to D$, and let $G \subset D\times D \subset \C^2$ be its graph. We equip $\C^2$ with the symplectic form $\omega_1\ominus\omega_2$, where $\omega_1$ and $\omega_2$ are the pullbacks of the standard area form in $\C=\R^2$  to the first and the second summands. Let $\C\cong\Delta \subset \C^2$ be the diagonal, the graph of the identity map. Graphs of symplectomorphisms are Lagrangian submanifolds.

Let $L: (z_1,z_2)\mapsto (w_1,w_2)$ be the complex linear map $\C^2\to \C^2$ given by the formulas
$$
w_1=\frac{z_1+z_2}{2},\ w_2=i(z_1-z_2).
$$
We denote $z_j=x_j+iy_j, w_j=u_j+iv_j$, $j=1,2$. 

\begin{lemma} \label{lm:turn}
The map $L$ is a symplectomorphism between $(\C^2,\omega_1\ominus\omega_2)$ and $(T^*\Delta\cong\C^2, d\lambda)$ where $\lambda=u_2du_1+v_2dv_1$.
\end{lemma}

\begin{proof}
The required equality
$$
dx_1\wedge dy_1 - dx_2\wedge dy_2 = du_2\wedge du_1 + dv_2\wedge dv_1
$$
is verified by a direct calculation. 
%
%
%
\end{proof}


Consider a symplectomorphism $F$ whose graph $G\subset\C^2$ is transverse to the fibers of the cotangent bundle $T^*\Delta$, that is, $G$ needs to be transverse to the projection $\pr:\C^2\cong T^*\Delta\to\Delta$ which in the identification of Lemma \ref{lm:turn} simply is the map 
\begin{equation*} \label{eq:proj}
(z_1,z_2)\mapsto \pr(z_1,z_2)=\left(\frac{z_1+z_2}{2},\frac{z_1+z_2}{2}\right).
\end{equation*}
We call such symplectomorphisms {\it moderate}.

In particular, $F$ is moderate if it is $C^1$-close to the identity.  More precisely, one has the next result. Let the symplectomorphism of $D$ be given by $(x,y)\mapsto F(x,y)=(f(x,y),g(x,y))$, in particular, $f_xg_y-f_yg_x=1$. 

\begin{lemma} \label{lm:trans}
The graph $G$ of the symplectomorphism $F$ is transverse to the fibers of $T^*\Delta$ near a point $(z,F(z))\in G$ if and only if $2+ f_x(z) + g_y(z)\ne 0$ with $z=x+iy$.
In particular, $F$ is moderate if and only if $2+ f_x + g_y\ne 0$ for all $(x,y)\in D$.
\end{lemma}

\begin{remark}
{\rm
In other words the graph $G$ a symplectomorphism $F$ is transverse to the fibers of $T^*\Delta$ if its Jacobian in $\mathrm{SL(2,\R)}$ is never negative parabolic. It easy to see that, e.g.,~the graph of a shear $(x,y)\mapsto (-x\pm y,-y)$ indeed is not transverse to the fibers.
}
\end{remark}

\begin{proof}
One has $G=\{(x,y,f,g)\mid (x,y)\in D\}$, and its tangent space is spanned by the vectors $(1,0,f_x,g_x)$ and $(0,1,f_y,g_y)$. The kernel of the projection on the diagonal is spanned by the vectors $(1,0,-1,0)$ and $(0,1,0,-1)$. These two spaces are not transverse if and only if
$$
\det \begin{pmatrix}
	1&0&f_x&g_x\\
	0&1&f_y&g_y\\
	1&0&-1&0\\
	0&1&0&-1
\end{pmatrix} = 0.
$$
This determinant equals $(1+f_x)(1+g_y)-f_yg_x = 2+f_x+g_y$,  and the result follows. 
\end{proof}


Let $\bar D\subset \Delta$ be the diagonal embedding of $D$, i.e., the graph of $\mathrm{id}_D$, and denote by $\bar \g$ its boundary.

\begin{lemma}\label{lm:moderate_and_domain}
Assume that $F$ is moderate and is the identity on the boundary $\g$. Then $\pr:G\to \bar D$ is a diffeomorphism. 
\end{lemma}

\begin{proof} 
That $\pr$ is a mop onto $\bar D$ is a fairly general topological argument and does not use that $G$ is transverse to the fibers. 

Indeed, assume that $B\subset \R^n$ is smooth $n$-dimensional manifold with boundary homeomorphic to the closed unit ball and $\Phi:B\to\R^n$ is a continuous map with $\Phi|_{\partial B}=\mathrm{id}$. We claim that $B\subset\Phi(B)$ holds. 

Assume for a contradiction that we find a point $p\in B\setminus \Phi(B)$. Then we may consider the map $\varphi:\partial B\hookrightarrow B\stackrel{\Phi}{\to}\R^n\setminus\{p\}\stackrel{r}{\simeq}\partial B$ where the map $r:\R^n\setminus\{p\}\to\partial B$ is a retraction. Here we use that $p\in B$, in fact, $p\in \mathring{B}$ since $\Phi|_{\partial B}=\mathrm{id}$. Moreover, $\Phi|_\g$ is not null-homotopic again since $\Phi|_{\partial B}=\mathrm{id}$. On the other hand $\Phi$ factors through the contractible set $B$, i.e.,~is null-homotopic, and thus so is $\Phi|_\g$. This contradiction shows that $B\subset \Phi(B)$. 

Applying this to $\Phi:B:=\bar D\to\C$ with $\Phi(z,z):=\pr(z,F(z))$ we conclude $\bar D\subset \pr(G)=\Phi(\bar D)$.

Now assume that $F$ is in addition transverse to $\pr$. Then $\pr$ is a local diffeomorphism. 
Let us prove that $\pr(G)\subset \bar D$. This would imply that $\pr$ is a covering, but since $\bar D$ is simply connected, it must be a diffeomorphism.

Choose a point $(w,w)\in\pr(G)\setminus \bar D\subset \Delta$ having maximal distance to some fixed point $o\in\bar D$. Choose $(z,F(z))\in G$ with $\pr(z,F(z))=(w,w)$. Since $\pr$ is a local diffeomorphism the point $(z,F(z))$ needs to be on the boundary of $G$: otherwise we could move its projection $(w,w)=\pr(z,F(z))$ further away from $o$. Since $G$ is diffeomorphic to $D$ and $F$ is the identity on $\partial  D$ we actually conclude $(z,F(z))=(z,z)=(w,w)\in\partial\bar D$. 

In conclusion, the points of $\pr(G)$ of maximal distance to $o$ are on the boundary of $\bar D$, in particular, $\pr(G)\subset \bar D$, as claimed.
\end{proof}


\begin{remark}
{\rm
We point out that in the previous proof of $\pr(G)\subset \bar D$ we did neither use that $F$ is a symplectomorphism nor that $D$ is simply connected. Moreover, this statement has the following geometric interpretation. The mid-point map $D\ni z\mapsto \tfrac12(z+F(z))\in\C$ takes values inside $D$ if $F$ is moderate and the identity on the boundary. For non-convex domains $D$ this is geometrically not directly clear to us.
} 
\end{remark}

We are ready to extend Theorem \ref{thm:two} to moderate symplectomorphisms.

\begin{theorem} \label{thm:mod}
Let $F$ be a moderate symplectomorphism of $D$ which is the identity on the boundary $\g$. Assume that $F$ has no conformal points on $\g$. Then $F$ possesses at least two conformal points inside $D$, counted with multiplicity.
\end{theorem}

\begin{proof}
A point $(x,y)$ is conformal if and only if
$$
f_x(x,y)=g_y(x,y),\ f_y(x,y)=-g_x(x,y).
$$
We claim that, counted with multiplicity, $f_y+g_x$ and $g_y-f_x $ vanish simultaneously at at least two points. We shall reduce this to case ii) in Theorem \ref{thm:two}.

In the notation of Lemma \ref{lm:turn}, $x=x_1, y=y_1$ and, accordingly, we write $u$ for $u_1$ and $v$ for $v_1$. Thus $(u,v)$ are coordinates in $\bar D$, and Lemma \ref{lm:moderate_and_domain} implies that $(x,y)\mapsto (u,v)$ is a diffeomorphism. Therefore we may consider 
$f_x(x,y)-g_y(x,y)$ and $f_y(x,y)+g_x(x,y)$ as functions of $u$ and $v$. We ``pack" these two functions in a vector filed
$$
V=(f_x(x,y)-g_y(x,y))\ \partial u + (f_y(x,y)+g_x(x,y))\ \partial v.
$$
Our goal is to show that $V$ has two zeros inside $\bar D$. For this purpose, we calculate $V$ on the boundary $\bar \g$.

The graph $G$ of $F$ is Lagrangian and, by assumption, transverse to the fibers of $T^*\Delta$. Thus the graph is a section of the cotangent bundle given by  the differential of a function $H:\bar D\to \R$. Since $F=(f,g)$ is the identity on $\g$, one has $dH=0$ on $\bar\g$. 

Using the coordinate change from Lemma \ref{lm:turn} we obtain the following formulas involving the graph $G=\{(x,y,f,g)\mid (x,y)\in D\}$ and $dH$ 
\begin{equation}\nonumber
\begin{aligned}
2x&=2u+H_v,\quad&2f&=2u-H_v,\\
2y&=2v-H_u,\quad &2g&=2v+H_u.
\end{aligned}
\end{equation}
%
Taking differentials, it follows that
\begin{equation}\nonumber
\begin{aligned}
2dx&=2du + H_{uv} du + H_{vv} dv,\quad &2f_x dx + 2f_y dy &= 2du - H_{uv} du - H_{vv} dv,\\
2dy&=2dv - H_{uu} du - H_{uv} dv,\quad  &2g_x dx + 2g_y dy&=2dv + H_{uu} du + H_{uv} dv.
\end{aligned}
\end{equation}
Substitute the two left equations into the right ones and equate the resulting 1-forms to obtain 
\begin{equation}\nonumber
\begin{aligned}
f_x \big(4- H_{uv}^2+H_{uu}H_{vv}\big)&= 4+H_{uv}^2-H_{uu}H_{vv}-4H_{uv},\\
f_y \big(4- H_{uv}^2+H_{uu}H_{vv}\big)&= -4H_{vv},\\
g_x \big(4- H_{uv}^2+H_{uu}H_{vv}\big)&= 4H_{uu},\\
g_y \big(4- H_{uv}^2+H_{uu}H_{vv}\big)&=4+H_{uv}^2-H_{uu}H_{vv}+4H_{uv}.
\end{aligned}
\end{equation}
Denote by $t$ the variable along $\bar\g$. Then deriving $H_u=0$ resp.~$H_v=0$ along $\bar\g$ we obtain from the chain rule
\begin{equation}\nonumber
\begin{aligned}
H_{uu}u_t+H_{vu}v_t&=0,\ 
H_{uv}u_t+H_{vv}v_t&=0.
\end{aligned}
\end{equation}
It follows that $H_{uu}H_{vv}=H_{uv}^2$, and we obtain
\begin{equation}\nonumber
\begin{aligned}
f_y+g_x = H_{uu}-H_{vv},\quad g_y-f_x = 2H_{uv}
\end{aligned}
\end{equation}
along $\bar\g$. 
Therefore, along $\bar\g$, one has
$$
V=(H_{uu}-H_{vv})\ \partial u + 2H_{uv}\ \partial v,
$$ 
which takes us to the second case of Theorem \ref{thm:two}. Its proof shows exactly the desired assertion.
\end{proof}

\section{Carath\'eodory conjecture and complex points of Lagrangian surfaces} \label{sec:compl}

\subsection{Geometry of the space of oriented lines} \label{subsec:lines}

In this section we survey the geometry of the space of oriented lines in $\R^3$ and the space of oriented non-parameterized geodesics in $S^3$ and $H^3$. We refer to \cite{AGK,GK1,GK2,GK3} and Section 5.6 of \cite{OT} for the material of this and the subsequent sections.

To start with, the space ${\mathcal G}(M)$ of oriented non-parameterized geodesics of a Riemannian manifold $M$ carries a symplectic structure that is obtained from the canonical symplectic structure of $T^*M$ by symplectic reduction.  This is true locally, and if the space of geodesics is a smooth manifold, then this manifold is symplectic. 

If $N\subset M$ is a cooriented hypersurface, then the normal geodesics to $N$ provide a Lagrangian immersion of $N$ into ${\mathcal G}(M)$. Denote the image of this immersion by ${\mathcal L}(N)$.

In the Euclidean case, one has ${\mathcal G}(\R^n)=T^*S^{n-1}$, in particular, ${\mathcal G}(\R^3)=T^*S^{2}$. For hyperbolic 3-space, the space of geodesics is also symplectomorphic to $T^*S^{2}$. If $M=S^3$, then the geodesics are the great circles, and ${\mathcal G}(S^3)= G_+(2,4)$, the Grassmannian of oriented 2-dimensional subspaces in $R^4$. One has: $G_+(2,4)=S^2\times S^2$, see, e.g.,  \cite{GW}.

Given a line $\ell \subset \R^3$, its train $S(\ell) \subset {\mathcal G}(\R^3)$ consists of the lines that intersect $\ell$. The train is a singular hypersurface. 
In the tangent space $T_\ell {\mathcal G}(\R^3)$, one obtains a quadratic cone. A similar construction works for  $H^3$ and $S^3$.

Next, ${\mathcal G}(\R^3)$ carries a complex structure. Let $\ell$ be an oriented line in $\R^3$. The $90^\circ$ rotation of the ambient space about $\ell$ induces an automorphism of the tangent space $T_\ell {\mathcal G}(\R^3)$; one obtains an almost complex structure, and this structure is integrable. Similarly for $H^3$ and $S^3$. 

Note that this complex structure is not compatible with the symplectic structure. For example, the lines through a fixed point form a Lagrangian sphere which is also a complex curve.

Furthermore, ${\mathcal G}(\R^3)$ carries a K\"ahler structure whose metric has signature $(+,+,-,-)$. Given a surface $N\subset \R^3$, the restriction of this metric to ${\mathcal L}(N)$ is either Lorentz or null. The latter happens at the complex points of ${\mathcal L}(N)$. 

The intersection of the quadratic cone in $T {\mathcal G}(\R^3)$ with  $T {\mathcal L}(N)$ is the light cone of the restriction of the K\"ahler metric to 
${\mathcal L}(N)$. The two light directions on ${\mathcal L}(N)$ correspond to the two principal directions on the surface $N$.  
 Likewise for $H^3$ and $S^3$. 

\subsection{Umbilics as conformal points} \label{subsec:umb}


The differential equation of the principal directions is given by
$$
\beta_{xy} dx^2 + (\beta_{yy}-\beta_{xx}) dxdy - \beta_{xy}dy^2=0,
$$
where $\beta(x,y)$ is  the Bonnet function, defined in a special Bonnet chart on the surface, see \cite{Da}. 
Representing the principal directions as the eigendirections of a quadratic form is similar to our alternative proof of Theorem \ref{thm:two} using the Hessian matrix $\mathit{Hess}(H)$.


Let $H:S^2 \to \R$ be a smooth function. Consider the respective Hamiltonian vector field on the sphere. 

\begin{lemma} \label{lm:confpt}
A point $s\in S^2$ is conformal for the Hamiltonian vector field of $H$  if and only if the second jet $j^2 H(s)$ equals the second jet of a first spherical harmonic, that is, the restriction to $S^2$ of an affine function $d+ax+by+cz$ in $\R^3$. 
\end{lemma}

\begin{proof}
A point is conformal for an area preserving diffeomorphism if its derivative sends circles to congruent circles, that is, it is an isometry. An orientation preserving isometry of $S^2$ is a rotation about an axis, and its Hamiltonian is the restriction to $S^2$ of an affine function $d+ax+by+cz$ (the axis being parallel to the vector $(a,b,c)$).  
\end{proof}

Recall that the support function of a convex surface $N\subset \R^3$ is a function $H:S^2 \to \R$ that equals the (signed) distance $H(s)$ from the origin to the tangent plane of $N$ whose oriented normal is $s\in S^2$. The differential $dH$ defines a Lagrangian section of $T^* S^2$; using its identification with ${\mathcal G}(\R^3)$, we obtain the Lagrangian surface ${\mathcal L}(N)$.

\begin{lemma} \label{lm:umb}
The umbilic points of $N$ correspond to points $s\in S^2$ where  $j^2 H(s)$ is the second jet of a first spherical harmonic. 
\end{lemma}

\begin{proof}
The umbilic points are the points where a sphere is 2nd order tangent to $N$, that is, where $j^2 H$ equals the second jet of the support function of a sphere. But the support functions of the spheres are the first harmonics (with $d>0$ being the radius of the sphere).
\end{proof}

Lemmas \ref{lm:confpt} and \ref{lm:umb}  imply that the umbilic points of a convex surface $N$ are the conformal points of the Hamiltonian vector field of the support function of $N$.
We are led to the following generalization of the Carath\'eodory conjecture.

\begin{conjecture} \label{con:4}
{\it An area preserving diffeomorphism of $S^2$ possesses at least two distinct conformal points.} 
\end{conjecture}

\begin{example}
{\rm Consider the particular case of Theorem \ref{thm:twoR} for the spherical metric. 
Assume that the support function has vanishing differential along a curve $\g \subset S^2$. Then the respective surface $N$ is tangent to a sphere along  $\g$. The normal  lines to $N$ along $\g$ form a cone, a developable surface. 

It is a classical fact 
that the normals to a surface along a curve form a developable surface if and only if the curve is a line of curvature, see, e.g., \cite{St}.
(Indeed, the normal lines at infinitesimally close points of a line of curvature intersect, and the locus of these intersection points is a curve in $\R^3$. Then the surface comprising the normals is the tangent developable of this space curve, a developable surface.) 

It follows that $\g$ is a closed line of curvature of $N$, and hence there there is an umbilic point in each of the two components of its complement.
}
\end{example}

Let us also interpret umbilics in terms of the space of oriented lines. This interpretation works equally well in the spherical and hyperbolic geometries. Let $N\subset \R^3$ be a cooriented surface.

\begin{lemma} \label{lm:norm}
The umbilic points of $N$ correspond to the complex points of ${\mathcal L}(N) \subset {\mathcal G}(\R^3)$.
\end{lemma}

\begin{proof}
A point of $N$ is umbilic if and only if the quadratic surface that approximates $N$ at this point up to the second derivatives is a sphere, say $S$. That is, ${\mathcal L}(N)$ and ${\mathcal L}(S)$ are tangent at the respective point. But ${\mathcal L}(S)$ is a complex curve in ${\mathcal G}(\R^3)$, and the result follows.
\end{proof}

Thus (a slightly generalized) formulation of the Carath\'eodory conjecture is that a Lagrangian section of $T^* S^2$ possesses at least two distinct complex points.

\begin{remark} \label{rmk:comp}
{\rm There exists a substantial literature concerning complex points of real surfaces. For example, consider an oriented closed surface $M$ embedded in $\C^2$. The generic complex points of $M$ are classified into four types: positive or negative (the orientation of the tangent plane is complex or not), and elliptic or hyperbolic, see \cite{CS,Bi} and \cite{Be}, pp. 137--139. One has the topological relations:
$$
e_+ -h_+=e_- - h_- = \frac{1}{2} \chi(M),
$$
where $\chi(M)$ is the Euler characteristic. In particular, if $M\subset \C^2$ is a sphere in general position, then there exist at least two elliptic complex points.}
\end{remark}

\bigskip

{\bf Acknowledgements}. Many thanks to D. Panov and A. Petrunin for useful discussions. 

PA acknowledge funding by the Deutsche Forschungsgemeinschaft (DFG, German Research Foundation) through Germany’s Excellence Strategy EXC-2181/1 - 390900948 (the Heidelberg STRUCTURES Excellence Cluster), the Transregional Colloborative Research Center CRC/TRR 191 (281071066). ST was supported by NSF grant DMS-2005444 and by a Mercator fellowship within the CRC/TRR 191, and he thanks the Heidelberg University for its hospitality.


\begin{thebibliography}{99}

\bibitem{AGK} D. Alekseevsky, B. Guilfoyle, W. Klingenberg. {\it On the geometry of spaces of oriented geodesics.}  Ann. Global Anal. Geom. {\bf 40} (2011), 389--409. 


\bibitem{Be} D. Bennequin. {\it Entrelacements et \'equations de Pfaff.} Ast\'erisque, 107--108, Soc. Math. France, Paris, 1983.

\bibitem{Bi} E.  Bishop. {\it Differentiable manifolds in complex Euclidean space.} 
Duke Math. J. {\bf 32} (1965), 1--21.


\bibitem{CS} S.-S. Chern, E. Spanier. {\it A theorem on orientable surfaces in four-dimensional space.} Comment. Math. Helv. {\bf 25} (1951), 205--209. 

\bibitem{Da} G. Darboux. {\it Sur la forme des lignes de courbure dans la voisinage d'un ombilic}. Note 7, 
Le\c{c}ons sur la Th\'eorie g\'en\'erale des Surfaces IV. Gauthier-Villars, Paris, 1896.

\bibitem{GH} M. Ghomi, R. Howard. {\it Normal curvatures of asymptotically constant graphs and Carath\'eodory's conjecture.} Proc. Amer. Math. Soc. {\bf 140} (2012),  4323--4335. 

\bibitem{GW} H. Gluck, F. Warner. {\it Great circle fibrations of the three-sphere.} Duke Math. J. {\bf 50} (1983), 107--132.

\bibitem{GK1} B. Guilfoyle, W. Klingenberg. {\it On the space of oriented affine lines in $\R^3$}. 
Arch. Math. {\bf 82} (2004), 81--84. 

\bibitem{GK2} B. Guilfoyle, W. Klingenberg. {\it Generalised surfaces in $\R^3$}.  Math. Proc. R. Ir. Acad. {\bf 104A} (2004), 199--209. 

\bibitem{GK3} B. Guilfoyle, W. Klingenberg. {\it An indefinite K\"ahler metric on the space of oriented lines.} J. London Math. Soc. {\bf 72} (2005), 497--509. 

\bibitem{GS} C. Gutierrez, J. Sotomayor. {\it Lines of curvature, umbilic points and Carath\'eodory conjecture.} Resenhas {\bf 3} (1998), 291--322.

\bibitem{La} H.F. Lai. {\it Characteristic classes of real manifolds immersed in complex manifolds.} Trans. Amer. Math. Soc. {\bf 172} (1972), 1--33. 

\bibitem{Lo} C. Loewner. {\it A topological characterization of a class of integral operators}.  Ann. of Math. {\bf 49} (1948), 316--332


\bibitem{OT} V. Ovsienko, S. Tabachnikov. {\it Projective differential geometry old and new. 
From the Schwarzian derivative to the cohomology of diffeomorphism groups. }
Cambridge Univ. Press, Cambridge, 2005.


\bibitem{SX} B. Smyth, F. Xavier. {\it Real solvability of the equation $\partial_{\bar z}^2 = \rho g$ and the topology of isolated umbilics}.  J. Geom. Anal. {\bf 8} (1998), 655--671.

\bibitem{SX1} B. Smyth, F. Xavier. {\it Eigenvalue estimates and the index of Hessian fields.}
Bull. London Math. Soc. {\bf 33} (2001), 109--112.

\bibitem{SG} J. Sotomayor, R. Garcia. {\it  Lines of curvature on surfaces, historical comments and recent developments.} 
S\~ao Paulo J. Math. Sci. {\bf 2} (2008), 99--143. 

\bibitem{St} D. Struik. {\it Lectures on Classical Differential Geometry.}  Addison-Wesley Press, Inc., Cambridge, Mass., 1950.

\bibitem{Tit} C. Titus. {\it A proof of a conjecture of Loewner and of the conjecture of Caratheodory on umbilic points.} Acta Math. {\bf 131} (1973), 43--77.

\bibitem{MO1} \url{https://mathoverflow.net/questions/354451}.

\end{thebibliography}
\end{document}